\documentclass{amsart}

\usepackage{amsthm,amsmath,mathtools,amssymb,comment,color}
\usepackage{amsfonts}
\usepackage{xcolor,soul}

\newtheorem{theorem}{Theorem}
\newtheorem{lemma}[theorem]{Lemma}
\newtheorem{proposition}[theorem]{Proposition}
\newtheorem{corollary}[theorem]{Corollary}
\theoremstyle{definition}

\calclayout

\title{A Reidemeister type theorem for petal diagrams of knots}
\author{Leslie Colton, Cory Glover, Mark Hughes, Samantha Sandberg}

\begin{document}
\maketitle
\begin{abstract}
We study petal diagrams of knots, which provide a method of describing knots in terms of permutations in a symmetric group $S_{2n+1}$. We define two classes of moves on such permutations, called trivial petal additions and crossing exchanges, which do not change the isotopy class of the underlying knot. We prove that any two permutations which represent isotopic knots can be related by a sequence of these moves and their inverses.
\end{abstract}

\section{Introduction}
Let $K$ be an oriented knot in $S^3$ or $\mathbb{R}^3$.  A \emph{diagram} for $K$ is a projection of $K$ to the plane with crossing information assigned to each multiple point of the projection.  Traditionally the knot $K$ is assumed to be positioned generically with respect to this projection, so that the only multiple points occuring are transverse double points.  In \cite{adams2015knot} Adams et al.\ introduce and study a new type of knot diagram, called a \emph{petal diagram}, whose corresponding projection contains a unique multiple point, and away from this crossing the diagram consists of $2n+1$ loops, called \emph{petals}, none of which are nested.  The strands passing through the multicrossing point are labelled with distinct integers $0,1, \ldots, 2n$ as in Figure~\ref{fig:petaldiagram} to indicate their relative heights with respect to the projection.  The authors of \cite{adams2015knot} prove that every knot can be isotoped with respect to a given projection so that its image is a petal diagram, and give an algorithm for realizing this projection.  They thus obtain a method of describing knots by specifying an element $\sigma$ of the symmetric group $S_{2n+1}$ on $2n+1$ letters, which is obtained from the cyclic ordering of the strand heights as we traverse a petal diagram in the counter-clockwise direction (see Figure~\ref{fig:petaldiagram}).  We call $\sigma$ a \emph{petal permutation} for the knot $K$.

\begin{figure}
\includegraphics[width=0.8\textwidth]{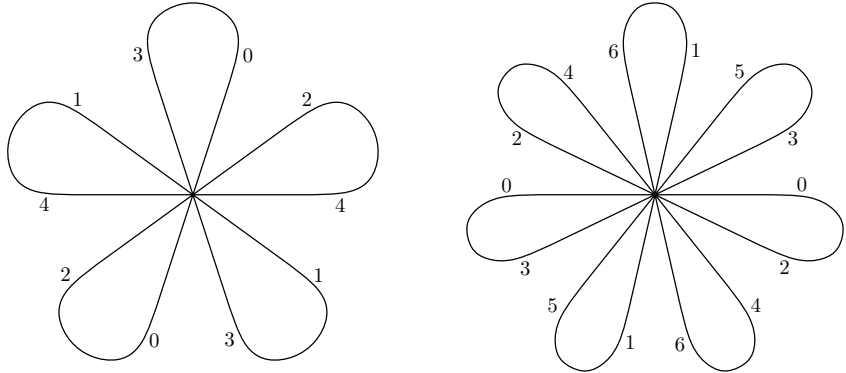}
\caption{Petal diagrams of the trefoil and figure-eight knots, with petal permutations $\sigma=(31420)$ and $\sigma'=(6420351)$ respectively. }
\label{fig:petaldiagram}
\end{figure}

Petal diagrams give rise to a new knot invariant, called the \emph{petal number}, which is introduced and studied in \cite{adams2015knot}, and where its relationship to more classical invariants are examined.  Petal diagrams also form the basis for a new model of random knots, called the \emph{Petaluma model}, which possesses several favorable properties.  See \cite{even2017models,even2018distribution,even2016invariants} for a description and study of this model.

On the level of isotopy classes this description of a knot by an element of $S_{2n+1}$ is not unique.  We address this nonuniqueness with the following theorem.  The definitions of trivial petals additions, trivial petal deletions, and crossing exchanges are given in Section~\ref{sec:modifyingpetalpermutations}. Collectively we refer to the set of these moves and their inverses as \emph{petal moves}.

\begin{theorem}
\label{thm:maintheorem}
Let $\sigma \in S_{2n+1}$ and $\sigma' \in S_{2m+1}$ be petal permutations which represent isotopic knots.  Then $\sigma$ can be transformed into $\sigma'$ by a sequence of trivial petal additions, trivial petal deletions, and crossing exchanges.
\end{theorem}

Theorem~\ref{thm:maintheorem} can be thought of as Reidemeister type theorem for the class of petal diagrams.  The proof of Theorem~\ref{thm:maintheorem} will be organized as follows.  In Section~\ref{sec:modifyingpetalpermutations} we describe trivial petal additions (along with their inverses, trivial petal deletions) and crossing exchanges, which are used in the statement of Theorem~\ref{thm:maintheorem}.  In Section~\ref{sec:stemdiagramsandpermutations} we introduce stem diagrams and stem permutations, which provide a convenient alternate description of petal diagrams and petal permutations.  We then translate the various petal moves from Section~\ref{sec:modifyingpetalpermutations} to an equivalent set of moves on stem diagrams in Section~\ref{sec:modifyingpetaldiagrams}, before using them to prove Theorem~\ref{thm:maintheorem} in Section~\ref{ref:proofsection}.

\section{Modifying petal permutations}
\label{sec:modifyingpetalpermutations}

In this section we describe the operations of trivial petal addition, trivial petal deletion, and crossing exchanges, which were mentioned in the statement of Theorem~\ref{thm:maintheorem}.  Once we have defined stem diagrams in Section~\ref{sec:stemdiagramsandpermutations}, we will proceed to give geometric descriptions of these operations in Section~\ref{sec:modifyingpetaldiagrams}.  In what follows, all knots will be oriented, and we assume that all petal diagrams are oriented counterclockwise.   Let $\sigma = (p_0p_1 \cdots p_{2n})$ be a petal permutation for a petal diagram of a knot $K$.

\subsection{Adding and deleting trivial petals} 

For $m \in \mathbb{Z}$, define $g_m : \mathbb{Z} \rightarrow \mathbb{Z}$ by 
\[ g_m(a) = \begin{cases} 
      a & a < m  \\
      a+2 & a \geq m 
   \end{cases}.
\]

We say that the petal permutation $\sigma'$ is obtained from $\sigma$ by \emph{trivial petal addition} if for some $0 \leq j \leq 2n$ we have either 
\[
\sigma' = (g_m(p_0) \cdots g_m(p_j) m (m+1) g_m(p_{j+1}) \cdots g_m(p_{2n}))
\]
or 
\[
\sigma' = (g_m(p_0) \cdots g_m(p_j) (m+1)m g_m(p_{j+1}) \cdots g_m(p_{2n})).
\]
In other words, $\sigma'$ is obtained from $\sigma$ by inserting one of the pairs $m(m+1)$ or $(m+1)m$ into the permutation $\sigma$ and shifting the other entries accordingly.  Note that in the above definition we also allow the pairs $m(m+1)$ and $(m+1)m$ to be inserted at the end of the permutation, after $g_m(p_{2n})$.  We will refer to the inverse of a trivial petal addition as a \emph{trivial petal deletion}.

On the level of petal diagrams, adding a trivial petal transforms a diagram with $2n+1$ petals into a petal diagram with $2n+3$ petals as in Figure~\ref{fig:trivialpetaltopetaldiagram}.  We will see in Section~\ref{sec:modifyingpetaldiagrams} that two petal diagrams related by trivial petal addition or deletion represent the same knot type.

\begin{figure}[h]
\includegraphics[width=0.8\textwidth]{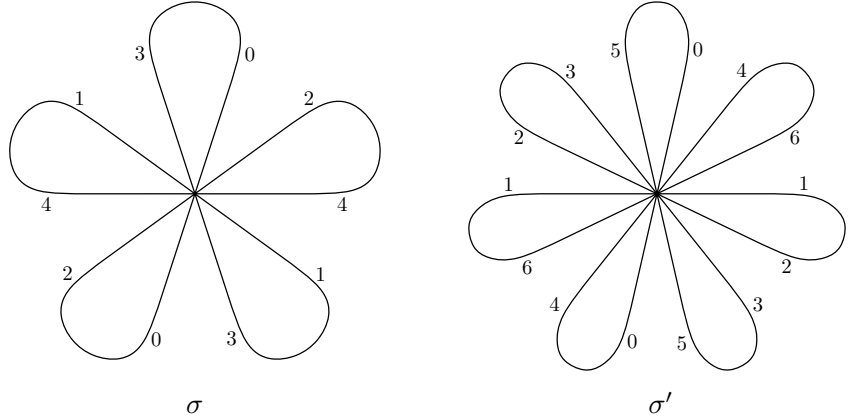}
\caption{The effect of adding a trivial petal to a petal diagram with petal permutation $\sigma=(31420)$.  Here we are adding a petal $(32)$ in between the 3 and 1 in $\sigma$.  This yields a petal diagram with petal permutation  $\sigma'=(5321640)$.}
\label{fig:trivialpetaltopetaldiagram}
\end{figure}

\subsection{Left- and right-pairs} To define the operation of crossing exchange, we must first define sets of left- and right-pairs of the petal permutation $\sigma = (p_0p_1 \cdots p_{2n})$.  Consider a word $W=p_0p_1 \cdots p_{2n}$ on the letters 
\[\{0,1, \ldots , {2n}\}=\{p_0,p_1,\ldots,p_{2n}\}.
\]
We think of $W$ as specifying an ordering on the elements of $\{0,1, \ldots , {2n}\}$ which agrees with the cyclic ordering defined by $\sigma$.  Notice first that $W$ is not uniquely determined by the permutation $\sigma$ since the word $p_0p_1 \cdots p_{2n}$ is only defined up to cyclic permutation.  Making a choice of $W$ is equivalent to selecting an axis of reflectional symmetry in the petal diagram ignoring the labels (see Figure~\ref{fig:leftrightpairs}).  Given such a choice of $W$ we define a \emph{set of left-pairs} for the petal permutation $\sigma = (p_0p_1 \cdots p_{2n})$ to be the set of tuples
\[
\mathcal{L}=\{ (p_0), (p_1, p_2), \ldots, (p_{2n-1},p_{2n})\},
\]
and a \emph{set of right-pairs} for $\sigma$ to be the set
\[
\mathcal{R}=\{ (p_0,p_1), (p_2,p_3), \ldots, (p_{2n})\}.
\]
Note that we will refer to all of the elements of $\mathcal{L}$ and $\mathcal{R}$ as left- and right-pairs respectively, even though $(p_0)$ and $(p_{2n})$ are not ordered pairs.  We call these special pairs \emph{basepoint pairs}.  The sets $\mathcal{L}$ and $\mathcal{R}$ are not uniquely defined by the petal permutation $\sigma$, but are uniquely defined by the word $W$.  In Figure~\ref{fig:leftrightpairs} we see that by starting at the top and moving counterclockwise around the diagram we first encounter petals labeled by the left-pairs, followed by petals labeled by the right-pairs.  Making a choice of left- and right-pairs is equivalent to selecting a starting point on a petal when traveling around in this way.

For any left- or right-pair $\Delta = (\delta, \delta')$, we call the set $E(\Delta)=\{ \delta, \delta'\}$ the \emph{set of endpoints} for the pair $\Delta$.  For example, the set of endpoints for the pair $(p_1,p_2)$  is $E((p_1,p_2))=\{p_1, p_2\}$, while we define the set of endpoints for the basepoint pair $(p_0)$ to be the singleton set $E((p_0))=\{p_0\}$.

\begin{figure}[h]
\includegraphics[width=0.35\textwidth]{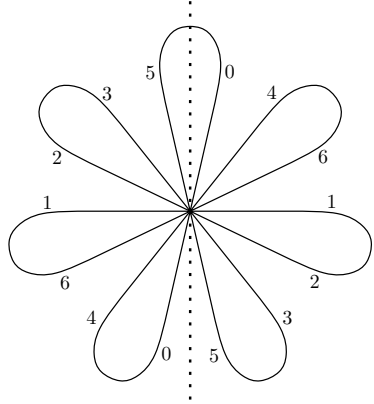}
\caption{Specifying $W$ as an axis of symmetry.  The choice of $W$ depicted corresponds to the set $\mathcal{L}=\{ (5), (32), (16),(40)    \}$ of left-pairs and the set $\mathcal{R}=\{ (53),(21),(64),(0)   \}$ of right-pairs.}
\label{fig:leftrightpairs}
\end{figure}

\subsection{Crossing exchanges}  
Suppose that $\mathcal{L}$ and $\mathcal{R}$ are a choice of sets of left- and right-pairs for the petal permutation $\sigma = (p_0p_1\cdots p_{2n})$.  Let $\Delta$ and $\Delta'$ be left-pairs (right-pairs respectively) with $E(\Delta) = \{m,w+1\}$ and $E(\Delta')=\{m+1,w\}$ for some distinct integers $m$ and $w$ with $w \geq m+2$.  Suppose furthermore that for any left-pair (right-pair respectively) $\Lambda$, the set of endpoints $E(\Lambda)$ is either contained in, or disjoint from the closed (possibly empty) interval $[m+2,w-1]$.  Then a petal permutation $\sigma'$ is said to be obtained from $\sigma$ by a \emph{crossing exchange} if a word representing $\sigma'$ can be obtained from a word $W$ representing $\sigma$ by switching the locations of $m$ and $m+1$, as well as the locations of $w$ and $w+1$ in $W$.  For example, if $\sigma$ is the permutation
\[
\sigma=(p_0 \cdots p_j m(w+1) p_{j+3} \cdots p_k w(m+1) p_{k+3} \cdots p_{2n}),
\]
then $\sigma'$ would be given by 
\[
\sigma=(p_0 \cdots p_j (m+1)w p_{j+3} \cdots p_k (w+1)m p_{k+3} \cdots p_{2n}).
\]
See Figure~\ref{fig:crossingexchangetopetaldiagram}.  Notice that in our definition of crossing exchanges, the basepoint pairs $(p_0)$ and $(p_{2n})$ cannot play the role of either $\Delta$ or $\Delta'$.

\begin{figure}[h]
\includegraphics[width=0.8\textwidth]{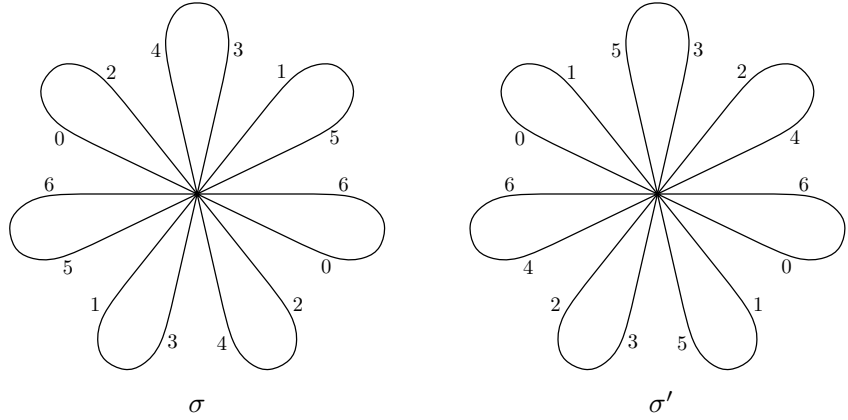}
\caption{The effect of applying a crossing exchange to a petal diagram with petal permutation $\sigma=(4206513)$.  A set of right-pairs for $\sigma$ is $\mathcal{R}=\{(42),(06),(51),(3)\}$.  The right-pairs $(w(m+1))=(42)$ and $((w+1)m)=(51)$ satisfy the conditions required to perform a crossing exchange.  The resulting petal diagram has petal permutation $\sigma'=(5106423)$.}
\label{fig:crossingexchangetopetaldiagram}
\end{figure}

\subsection{Example} We illustrate these moves using a simple example.  Consider the petal permutations $\sigma = (0351642)$ and $\sigma'= (135026478)$, which both represents the figure-eight knot.  We will relate these two petal permutations by petal moves.  We begin by adding a trivial petal $(01)$ in between the 1 and the 6 in $\sigma$ which gives, after shifting the other entries appropriately, $\sigma_1 = (257301864)$.  We then add a trivial petal $(23)$ between the 7 and 3 in $\sigma_1$ to give $\sigma_2= (47923501(10)86)$.

Applying a cyclic permutation to our expression of $\sigma_2$, we write $\sigma_2=(01(10)86479235)$. One choice of left-pairs for $\sigma_2$ then is the set
\[
\mathcal{L} = \left\{(0),(1(10)),(86),(47),(92),(35) \right\}.
\]
Notice that the left-pairs $(m(w+1)) = (1(10))$ and $(w(m+1))=(92)$ satisfy the hypotheses of the crossing exchange, and hence we can perform a crossing exchange by replacing them with $(29)$ and $((10)1)$ respectively.  This yields the petal permutation $\sigma_3 = (0298647(10)135)$.  Finally, we can perform a trivial petal removal to $\sigma_3$ by removing $(98)$ and shifting, to give $\sigma' = (026478135)=(135026478)$.

\section{Stem diagrams and permutations}
\label{sec:stemdiagramsandpermutations}

\subsection{Stem diagrams} Let $K$ be an oriented knot, with diagram $D$ in the plane $P$ whose only multiple points are transverse double points.  Let $\alpha$ be the oriented image of $\mathbb{R}$ in $P$ under a proper embedding, which intersects $D$ transversely and divides $P$ into two components, $L$ and $R$.  We call $\alpha$ an \emph{axis} for the diagram $D$, and we assume that $L$ is to the left of $\alpha$ and $R$ is to the right.  Furthermore we assume that $D \cap \alpha \neq \emptyset$, and that $\alpha$ does not intersect $D$ at any of its crossings.  Fix a choice of $s \in D \cap \alpha$ so that when traveling from $s$ along $D$ in the positively oriented direction, we first pass into $L$, before passing into $R$.

The axis $\alpha$ divides $D$ into a collection of immersed arcs, whose self-intersections and pairwise intersections are all transverse double points.  We call these immersed arcs in $L$ \emph{left-strands} of $D$, and the immersed arcs in $R$ \emph{right-strands} of $D$.    Starting at $s$ and traveling along $D$ in the positively oriented direction induces a natural ordering on the collection of left- and right-strands, which alternates between left-strands $\ell_i$ and right-strands $r_j$, and which we denote by $\ell_0, r_1, \ell_1, r_2, \ldots , \ell_n,r_{n+1}$.   

The triple $(D, \alpha, s)$ is called a \emph{stem diagram} for $K$ if, when traveling from $s$ in the positively oriented direction along $D$, we pass each crossing in $L$ along the under-strand first before returning along the over-strand, while each crossing in $R$ is encountered along the over-strand first before returning along the under-strand.  In terms of left- and right-strands this implies that each pair $(L,\ell_i)$ and $(R,r_j)$ is the diagram of an unknotted tangle, and for all $i > j$, any intersection between $\ell_i$ and $\ell_j$ has $\ell_i$ passing \emph{over} $\ell_j$, and any intersection between $r_i$ and $r_j$ has $r_i$ passing \emph{under} $r_j$.

%Two stem diagrams $(D, \alpha, s)$ and $(D', \alpha', s')$ are \emph{isotopic} if there is a planar isotopy taking the triple $(D, \alpha, s)$ to the triple $(D', \alpha', s')$.

\begin{lemma}
\label{lem:stemtopetal}
Let $K$ be a knot and $(D, \alpha, s)$ a stem diagram for $K$ as above.  Then after an isotopy of $K$ which preserves the diagram $D$, there is a projection of $K$ to a plane so that the image of $K$ is a petal diagram with $2n+1$ petals. 
\end{lemma}

\begin{proof}
The proof essentially proceeds by changing our perspective and viewing the stem diagram from the top.  More precisely, we begin by identifying the projection plane $P$ with $\mathbb{R}^2$, so that $\alpha$ corresponds to the $y$-axis, and so that $L$ and $R$ correspond to the left and right half-planes in $\mathbb{R}^2$ respectively.  After a diffeomorphism of $\mathbb{R}^3$ we may also assume that the projection from $\mathbb{R}^3$ to $\mathbb{R}^2$ is given by orthogonal projection $\pi$ to the $xy$-plane in the $z$-direction.

Let $\widetilde{\alpha}$ be the $y$-axis in $\mathbb{R}^3$, which we think of as a lift of $\alpha$ under the projection map.  By an isotopy of $K$ which changes only the $z$-coordinates, we may arrange $K$ so that each intersection point in $(\alpha \cap D) \backslash \left\{ s \right\}$ lifts to an intersection point of $\widetilde{\alpha}$ with $K$, and so that $K$ passes below $\widetilde{\alpha}$ at the point $\pi^{-1}(s) \cap K$.

Furthermore, this isotopy can be chosen so that the orthogonal projection of $K$ to the $xz$-plane will yield a petal diagram of $K$.  Indeed, if $(r, \theta)$ denote polar coordinates on the $xz$-plane, then $(r, \theta, y)$ gives a cylindrical coordinate system on $\mathbb{R}^3$, and the lift of each left-strand and right-strand can be arranged so that their interiors live in disjoint  $\theta$-intervals.  Then, because the lift of each $\ell_0,r_1, \ell_1, \ldots , r_n, \ell_n,r_{n+1}$ is an unknotted tangle, they can be arranged within their disjoint $\theta$-neighborhoods so that $r_1, \ell_1, \ldots , r_n, \ell_n$ each project to a simple loop in the $xz$-plane, while the union of the lifts of $\ell_0$ and $r_{n+1}$ project to a single simple loop.  
\end{proof}

Notice that up to planar isotopy the petal diagram constructed in the proof of Lemma~\ref{lem:stemtopetal} is unique.  We will therefore refer to it as the \emph{petal diagram associated to} $(D, \alpha, s)$.

It can similarly be seen that a petal diagram of $K$ gives rise to a (nonunique) stem diagram for $K$ by viewing the petal projection from its side.  More precisely, suppose that $K$ is arranged in $\mathbb{R}^3$ as above, with the petal diagram being given by projection to the $xz$-plane.  Suppose that the crossing of the petal projection is situated at the origin of the $xz$-plane, and so that away from the origin the diagram intersects the $z$-axis in only one other point, which we denote $s'$, on the negative $z$-axis.  Suppose furthermore that $n$ of the petals lie to the left of the $z$-axis, $n$ of the petals lie to the right of the $z$-axis, and one of the petals is split into two components by the negative $z$-axis at the point $s'$ (see Figure~\ref{fig:petaltostem}).  Let $\hat{s}$ be the lift of the point $s'$ to the knot $K$.  Then by projecting $K$ to the $xy$-plane we obtain a stem diagram of the knot $K$, where $\alpha$ is the $y$-axis in the $xy$-plane, and the basepoint $s$ is the image of $\hat{s}$.  As all knots admit petal diagrams \cite{adams2015knot}, every knot thus admits a stem diagram.

\begin{figure}
\includegraphics[width=0.85\textwidth]{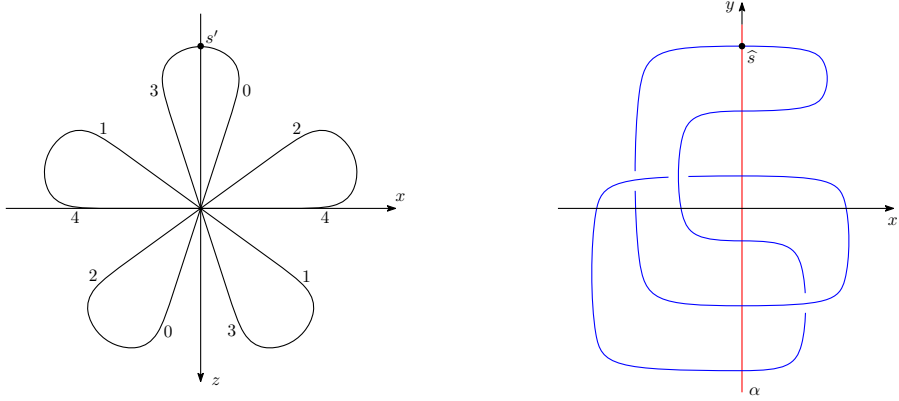}
\caption{Arranging a petal diagram in the $xz$-plane, so that the projection of $K$ to the $xy$-plane is a stem diagram.}
\label{fig:petaltostem}
\end{figure}

\subsection{Stem and petal permutations} \label{subsec:stemandpetal} Notice that we can associate an element $\tau$ of $S_{2n+2}$ to the stem diagram $(D, \alpha, s)$ as follows.  Begin by first labeling the points of $D \cap \alpha$ with the integers $0,1, \ldots, 2n+1$, in order \emph{from top to bottom}.  We call the label associated to $c \in D \cap \alpha$ the \emph{level} of $c$.  For example, if when traveling along $\alpha$ from top to bottom we encounter the points of $D \cap \alpha$ in order $c_0, c_1, \ldots, c_{2n+1}$, then we say that $c_0$ is at level 0, $c_1$ is at level 1, etc.

We then obtain $\tau$ from the stem diagram $(D, \alpha, s)$ by starting at $s$ and traveling in the positively oriented direction along $D$, recording the levels of the points in $D \cap \alpha$ as we pass them (see Figure~\ref{fig:stemdiagram}).  This gives an ordering of the integers $0,1, \ldots, 2n+1$, and defines a permutation $\tau \in S_{2n+2}$ which we call the \emph{stem permutation} associated to the stem diagram $(D, \alpha, s)$.  

\begin{figure}[h]
\includegraphics[width=0.28\textwidth]{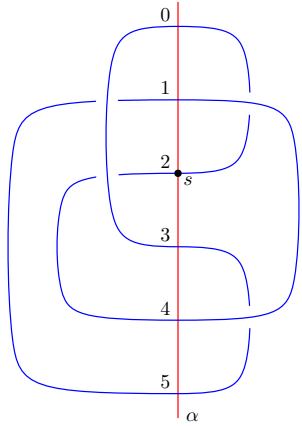}
\caption{A stem diagram for the trefoil knot, corresponding to the stem permutation $\tau=(241530)$ and the petal permutation $\sigma=(31420)$.}
\label{fig:stemdiagram}
\end{figure}

Given the stem permutation $\tau$ of a stem diagram ($D, \alpha, s)$ we can recover the petal permutation $\sigma$ of the associated petal diagram as follows.  Suppose that $\tau$ can be written as the permutation $\tau = ({t_0} {t_1}\cdots {t_{2n+1}})$, where $t_0$ is the level of the basepoint $s$.  For any $m \in \mathbb{Z}$ let $f_m : \mathbb{Z} \rightarrow \mathbb{Z}$ be defined by  
\[ f_m(a) = \begin{cases} 
      a & a\leq m  \\
      a-1 & a > m 
   \end{cases}.
\]
Then the petal permutation of the petal diagram associated to $(D, \alpha, s)$ is given by $\sigma  = ( f_{t_0} (t_1) f_{t_0} (t_2) \cdots f_{t_0}(t_{2n+1}))$.  In other words, we delete the level of the crossing $s$ from the permutation $\tau$, and shift all of the levels greater than that of $s$ by one so that we end up with an ordering of the integers $\{ 0,1, 2, \ldots , 2n\}$ instead of $\{ 0,1, 2, \ldots , 2n+1\}$.

Notice, however, that a petal permutation does not give rise to a unique stem permutation.  Indeed, when converting a petal diagram to a stem diagram there is no canonical choice of petal on which to place the basepoint $s$.  Such a choice is equivalent to selecting a word $W=p_0 p_1\cdots p_{2n}$ from the petal permutation $\sigma = (p_0 p_1\cdots p_{2n})$, and is also equivalent to selecting a choice of sets of left- and right-pairs $\mathcal{L}$ and $\mathcal{R}$ for $\sigma$.

Furthermore, even once a choice of basepoint has been fixed and a stem diagram drawn, because the left-strand $l_0$ and right-strand $r_{n+1}$ are the bottommost strands with respect to the projection in a stem diagram, they can be moved up or down and arranged via Reidemeister 2 moves to intersect $\alpha$ at any desired level.  See Figure~\ref{fig:changinglevel}.  %For notational convenience, and to avoid shifting entries in the petal permutation, we will often assume that $s$ is at the lowest level $2n+1$, so that the petal permutation $\sigma = (p_0 p_1\cdots p_{2n})$ corresponds to the stem permutation $\tau= ((2n+1)p_0 p_1\cdots p_{2n})$. 

Finally, we note that there is a one-to-one correspondence between the left-strands $\ell_0, \ell_1, \ldots , \ell_n$ in a stem diagram and the left-pairs $(p_0), (p_1,p_2),\ldots,$ $(p_{2n-1},p_{2n})$ respectively of the petal permutation, while the right-strands of the stem diagram $r_1, \ldots,$ $r_n,r_{n+1}$ correspond respectively to the right-pairs $(p_0,p_1), \ldots,(p_{2n-2},p_{2n-1}),$ $(p_{2n})$ of the petal permutation.  

\begin{figure}[h]
\includegraphics[width=0.85\textwidth]{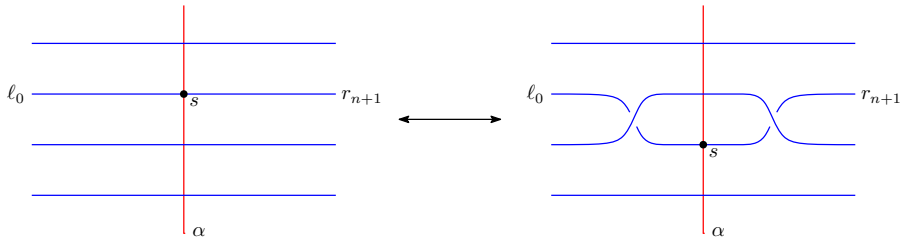}
\caption{Changing the level of the basepoint $s$ in a stem diagram.}
\label{fig:changinglevel}
\end{figure}

\section{Modifying stem diagrams}
\label{sec:modifyingpetaldiagrams}

\subsection{Petal-preserving isotopy} Let $K$ be a knot with stem diagram $(D, \alpha, s)$.  A \emph{petal-preserving isotopy} of $(D, \alpha, s)$ is a sequence of diagrams $D=D_1, D_2, \ldots,$ $D_m$ so that for each $1 \leq j \leq m$ the triple $(D_j,\alpha, s)$ is a stem diagram, and for $1 \leq j \leq m-1$ the diagram $D_{j+1}$ is obtained from $D_j$ by a either a planar isotopy supported away from $\alpha$, or a Reidemeister move which is contained in a disk away from $\alpha$.  Petal-preserving isotopies can be thought of roughly as the projection of isotopies in $\mathbb{R}^3$ which preserve the petal structure of $K$.  As petal-preserving isotopies leave a neighborhood of $\alpha$ fixed, they do not change the stem permutation of the stem diagram $(D, \alpha, s)$, and hence do not change the petal permutation of the associated petal diagram.  In fact, given $\alpha$ and $s$ the stem permutation uniquely determines the stem diagram up to petal-preserving isotopy.

\subsection{Reduced stem diagrams} We now define a special class of stem diagram, called a \emph{reduced stem diagram}.  Let $(D, \alpha, s)$ be a stem diagram where each left- and right-strand is a half-circle in the plane, connecting its endpoints on $D \cap \alpha$.  Then $(D, \alpha, s)$ is called a \emph{reduced stem diagram}.

\begin{lemma}
Every stem diagram can be converted into a reduced stem diagram by a petal-preserving isotopy.
\end{lemma}

\begin{proof}
Let $\ell=\ell_j$ be a left-strand.  Then because each self-crossing along $\ell$ is encountered first as an undercrossing, $(L, \ell)$ is the diagram of an unknotted tangle in the left half-plane $L$.  There is thus a sequence of Reidemeister moves and planar isotopies away from $\alpha$ which allow us to convert $\ell$ into a half-circle connecting its endpoints.  Because the left-strands are stacked so that $\ell_i$ passes over $\ell_j$ whenever $i > j$, these modifications can be performed without affecting the positioning of the other left-strands.  The proof for the right-strands is identical.
\end{proof}

Reduced stem diagrams are helpful because they allow us to identify crossings in a diagram by only knowing the stem permutation. Suppose that $\tau = (t_0 t_1 \cdots t_{2n+1})$ is a stem permutation, with $t_0$ denoting the level of the basepoint $s$.  Then for $0 \leq j \leq n$ the endpoints of the left-strand $\ell_j$ will be on levels $t_{2j}$ and $t_{2j+1}$.  Likewise for $1 \leq i \leq n$ the endpoints of the right-strand $r_i$ will be on levels $t_{2i-1}$ and $t_{2i}$, while the levels of the endpoints of $r_{n+1}$ will be $t_{2n+1}$ and $t_0$.  %%%%%%%%%%%%%%%%For notational convenience we will often identify the left- and right-strands by the ordered pair containing the level of their endpoints $(t_i,t_j)$.  TODO: relate to left- and right-pairs.

Suppose that $\Delta$ and $\Lambda$ are either two left-strands or two right-strands of a reduced stem diagram, with endpoints at levels $\delta,\delta'$ and $\lambda,\lambda'$ respectively.  Then $\Delta$ and $\Lambda$ will intersect transversely in a single point if and only if 
\[
(\delta-\lambda)(\delta'-\lambda)(\delta-\lambda')(\delta'-\lambda')<0
\]
and will be disjoint otherwise.  In other words, $\Delta$ and $\Lambda$ will intersect in a single point precisely when exactly one of the endpoints of $\Delta$ lies on $\alpha$ in between the endpoints of $\Lambda$.

\subsection{Trivial petals and stem diagrams} We now discuss how the addition or deletion of trivial petals to a petal permutation $\sigma$, as defined in Section~\ref{sec:modifyingpetalpermutations}, modifies a stem diagram of the knot.  Let $(D, \alpha, s)$ be a stem diagram, and let $\gamma$ be an arc embedded in the plane whose boundary $\partial \gamma$ consists of one point on $\alpha \backslash D$ and one point on $D \backslash \alpha$.  Given $\gamma$ we can modify the diagram $D$ by either of the moves shown in Figure~\ref{fig:trivialpetal}, yielding a new diagram we denote by $D'$.  As the interior of $\gamma$ may intersect $D$, we perform Reidemeister~2 moves at each of the intersection points in $D \cap \operatorname{int}{\gamma}$ when modifying the diagram.  There is a unique choice for each of these Reidemeister~2 moves (determined by the relative heights of the strands with respect to the projection), as well as the Reidemeister 1 move in the second diagram, so that the resulting triple $(D' , \alpha, s)$ is another stem diagram for the knot $K$.  In this case we say that $(D' , \alpha, s)$ is obtained from $(D , \alpha, s)$ by the addition of a \emph{trivial petal}.  Likewise we say that $(D, \alpha, s)$ is obtained from $(D' , \alpha, s)$ by the removal or deletion of a trivial petal.  Note that by our definition we cannot delete a trivial petal on a strand with an endpoint at the basepoint $s$.  Furthermore, the addition or deletion of a trivial petal to a stem diagram does not change the isotopy class of the associated knot.

\begin{figure}
\includegraphics[width=0.5\textwidth]{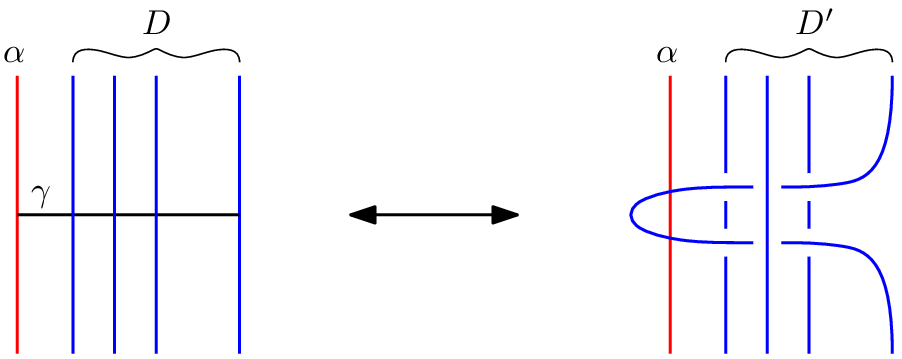} 
\includegraphics[width=0.5\textwidth]{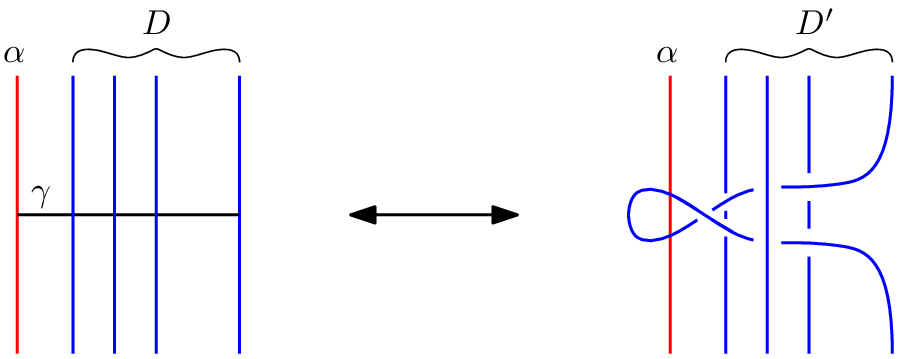}
\caption{Trivial petal addition/deletion to a stem diagram.}
\label{fig:trivialpetal}
\end{figure}

We first observe how the addition of a trivial petal changes the stem permutation of the stem diagram $(D, \alpha, s)$.  Suppose that $\tau = (t_0\cdots t_{2n+1})$ is the stem permutation of $(D, \alpha, s)$, with the basepoint $s$ being at level $t_0$.  Suppose that we are adding a trivial petal along the arc $\gamma$, where $\gamma$ has one endpoint on $\alpha$ between the points of $D \cap \alpha$ at levels $m-1$ and $m$, and the other endpoint is on the left- or right-strand of $D$ that connects the points at level $t_j$ and $t_{j+1}$.  Recall that $g_m : \mathbb{Z} \rightarrow \mathbb{Z}$ was defined above as
\[ g_m(a) = \begin{cases} 
      a & a < m  \\
      a+2 & a \geq m 
   \end{cases}.
\]
Then the stem permutation of the diagram $(D' , \alpha, s)$ obtained by adding a trivial petal along $\gamma$ will be either
\[
\tau' = (g_m(t_0)g_m(t_1)\cdots g_m(t_j )  m (m+1)     g_m(t_{j+1}) \cdots g_m(t_{2n+1}))
\]
or
\[
\tau' = (g_m(t_0)g_m(t_1)\cdots g_m(t_j )   (m+1)  m   g_m(t_{j+1}) \cdots g_m(t_{2n+1})))
\]
depending on the choice of $\gamma$, the relative orderings of $t_j,t_{j+1}$, and $m$, and the choice of move from Figure~\ref{fig:trivialpetal}.  

Notice that the petal permutations $\sigma$ and $\sigma'$ associated to $(D, \alpha, s)$ and $(D', \alpha, s)$ are then related by a single trivial petal addition, as defined in Section~\ref{sec:modifyingpetalpermutations}.  %Furthermore, any trivial petal addition or deletion to a petal permutation $\sigma$ can be realized by a move in Figure~\ref{fig:trivialpetal} applied to some stem diagram for $\sigma$.  Thus we see that trivial petal additions and deletions to $\sigma$ do not change the corresponding knot type. 
We record these facts for future use.

\begin{lemma}
\label{lem:trivialpetal}
Suppose that $\sigma$ and $\sigma'$ are petal permutations.  Then $\sigma'$ is related to $\sigma$ by the addition of a trivial petal if and only if there are stem diagrams $(D,\alpha,s)$ and $(D',\alpha,s)$ with associated petal permutations $\sigma$ and $\sigma'$ respectively, such that $(D',\alpha,s)$ is related to  $(D,\alpha,s)$ by the addition of a trivial petal.  
\end{lemma}

\begin{proof}
%Clearly if $(D',\alpha,s)$ is related to $(D , \alpha, s)$ by a trivial petal addition, then its associated petal permutation $\sigma'$ will be related to  $\sigma$ by a trivial petal addition.  Conversely, suppose $\sigma'$ is related to $\sigma$ by a trivial petal addition.  Then given a stem diagram $(D, \alpha, s)$ for $\sigma$, we can find an arc $\gamma$ so that adding a trivial petal along $\gamma$ yields a stem diagram with associated petal permutation $\sigma'$.
Suppose that $\sigma'$ is related to $\sigma$ by a trivial petal addition.  Then given a stem diagram $(D, \alpha, s)$ for $\sigma$, we can find an arc $\gamma$ so that by adding a trivial petal along $\gamma$ we obtain a stem diagram $(D', \alpha, s)$ with associated petal permutation $\sigma'$.  The converse follows from the observations in the preceding paragraph.
\end{proof}

\begin{corollary}
Suppose that $\sigma$ and $\sigma'$ are petal permutations, with $\sigma'$ related to $\sigma$ by a trivial petal addition.  Then $\sigma$ and $\sigma'$ represent the same knot type.
\end{corollary}

%\begin{lemma}
%\label{lem:trivialpetal}
%Let $\sigma$ and $\sigma'$ be petal permutations with associated stem diagrams $(D,\alpha,s)$ and $(D',\alpha,s)$ respectively.  Then $(D',\alpha,s)$ is related to $(D,\alpha,s)$ by the addition (deletion) of a trivial petal, if and only if $\sigma'$ is related to $\sigma$ by the addition (deletion) of a trivial petal.  Moreover, if $\sigma$ and $\sigma'$ are related by a trivial petal addition or deletion then they represent the same knot type.
%\end{lemma}

\subsection{Crossing exchanges and stem diagrams}
Consider now a stem diagram $(D, \alpha, s)$, and suppose that $\Delta$ and $\Delta'$ are two left-strands or two right-strands, with the boundary of $\Delta$ being a pair of points on $\alpha$ sitting at levels $m$ and $w+1$, and the boundary of $\Delta'$ being a pair of points sitting at levels $m+1$ and $w$ (here we assume without loss of generality that $w > m$).  Suppose that none of the endpoints of $\Delta$ or $\Delta'$ are at the basepoint $s$, and that all of the other strands on the same side of the axis as $\Delta$ and $\Delta'$ either have none or both of their endpoints contained between the levels $m+1$ and $w$.

Suppose now that we modify the diagram $D$ in a neighborhood of $\alpha$ by adding a pair of crossings as in Figure~\ref{fig:crossingexchangestem}.  Notice that neither $\Delta$ nor $\Delta'$ have an endpoint at the basepoint $s$, and that the signs of the crossings can be uniquely chosen so that the modification results in a stem diagram $(D', \alpha, s)$.  We say that $\Delta'$ is obtained from $\Delta$ by a \emph{crossing exchange}.  In Figure~\ref{fig:crossingexchangestem} we illustrate the case when $\Delta$ and $\Delta'$ are right-strands, and $\Delta$ passes under $\Delta'$ in the right half-plane.  To obtain the diagram when $\Delta$ passes over $\Delta'$ we reverse the crossings, and we obtain the corresponding diagrams when $\Delta$ and $\Delta'$ are left-strands by reflecting Figure~\ref{fig:crossingexchangestem} along $\alpha$.  We will refer to the reverse of each of these moves as crossing exchanges as well.  Indeed, if we perform the move twice to $(D,\alpha,s)$ along the same strands, we will obtain a diagram that is identical to $(D,\alpha,s)$ except for two pairs of cancelling crossings.  As suggested by the name, this operation corresponds to performing a crossing exchange on the associated petal permutation as defined in Section~\ref{sec:modifyingpetalpermutations}.

\begin{figure}
\includegraphics[width=0.7\textwidth]{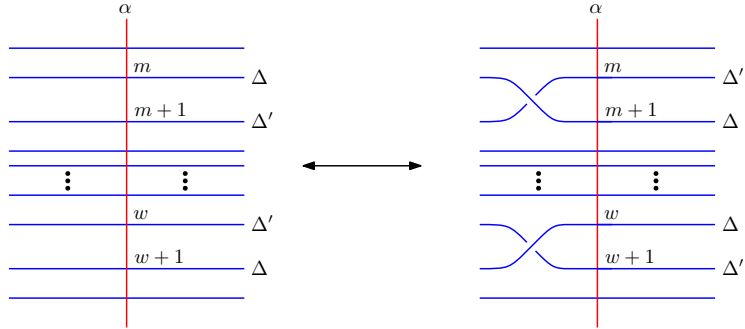} 
\caption{A crossing exchange on a stem diagram.}
\label{fig:crossingexchangestem}
\end{figure}

\begin{lemma}
\label{lem:crossingexchange}
Suppose that $\sigma$ and $\sigma'$ are petal permutations.  Then $\sigma'$ is related to $\sigma$ by a crossing exchange if and only if there are stem diagrams $(D,\alpha,s)$ and $(D',\alpha,s)$ with associated petal permutations $\sigma$ and $\sigma'$ respectively, such that $(D',\alpha,s)$ is related to  $(D,\alpha,s)$ by a crossing exchange.   Moreover, if $\sigma$ and $\sigma'$ are related by a crossing exchange then they represent the same knot type.
\end{lemma}

%\begin{lemma}
%\label{lem:crossingexchange}
%Let $\sigma$ and $\sigma'$ be petal permutations with associated stem diagrams $(D,\alpha,s)$ and $(D',\alpha,s)$ respectively.  Then $(D',\alpha,s)$ is related to $(D,\alpha,s)$ by a crossing exchange if and only if $\sigma'$ is related to $\sigma$ by a crossing exchange.  Moreover, if $\sigma$ and $\sigma'$ are related by a crossing exchange then they represent the same knot type.
%\end{lemma}

\begin{proof}
The proof of the first claim is similar to the proof of Lemma~\ref{lem:trivialpetal}.  When choosing a stem diagram for $\sigma$, however, we make sure to choose one whose left- and right-strands correspond to the left- and right-pair choices used in the the crossing exchange on $\sigma$.  We also reposition the basepoint $s$ by the move in Figure~\ref{fig:changinglevel}, so that either both or neither of the endpoints of the basepoint strand sit between the levels $m+1$ and $w$ as required.

\begin{figure}
\includegraphics[width=0.65\textwidth]{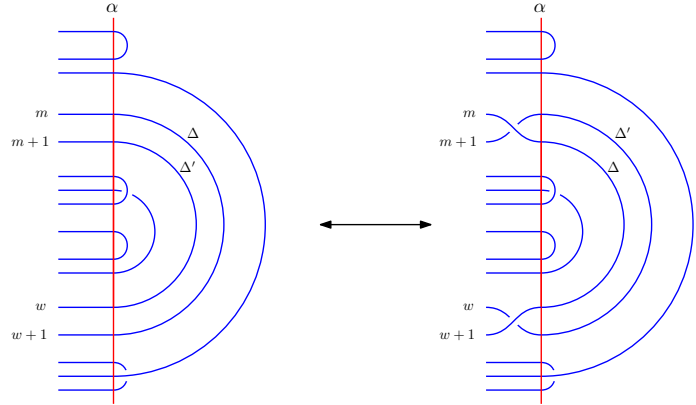} 
\caption{A crossing exchange performed on a reduced stem diagram.}
\label{fig:crossingexchangestem2}
\end{figure}

To prove that crossing exchanges preserve knot types, we note that because of our assumption that all of the strands besides $\Delta$ and $\Delta'$ have either both or neither of their endpoints contained in the interval $[m+2,w-1]$, there is a petal-preserving isotopy taking $(D, \alpha, s)$ to a reduced stem diagram as in the left-hand side of Figure~\ref{fig:crossingexchangestem2}.  More precisely, after converting $(D, \alpha, s)$ into a reduced stem diagram the strands $\Delta$ and $\Delta'$ will be concentric semi-circles which do not intersect any of the other strands.  It is clear then that performing a crossing exchange to such a diagram does not change the knot type.  After performing the crossing exchange, we can apply the reverse of our petal-preserving isotopy to obtain a stem diagram $(D' , \alpha, s)$, which agrees with $(D, \alpha, s)$ away from the pair of crossings which were introduced. 
\end{proof}

When performing crossing exchanges on stem diagrams in what follows, we will always apply them in the situation on the left-hand side of Figure~\ref{fig:crossingexchangestemr2}, in other words, when there is a crossing directly on the opposite side of the axis from $\Delta$ and $\Delta'$.  In this situation the crossing exchange will create a canceling pair of crossings, which can be removed by a Reidemeister 2 move.  The combined effect of this crossing exchange and Reidemeister move is to take the crossing between the strands at levels $m$ and $m+1$, and to move it to the strands at levels $w$ and $w+1$ as shown in Figure~\ref{fig:crossingexchangestemr2}.  When such a pair of canceling crossings is created during a crossing exchange and then removed, we will consider the Reidemeister 2 move to be part of the crossing exchange, and will refer to the combination of both of these moves as a crossing exchange.  As Reidemeister 2 moves away from the axis $\alpha$ do not change the associated stem or petal permutations, the conclusions of Lemma~\ref{lem:crossingexchange} continue to hold for our expanded definition of crossing exchange.

\begin{figure}
\includegraphics[width=\textwidth]{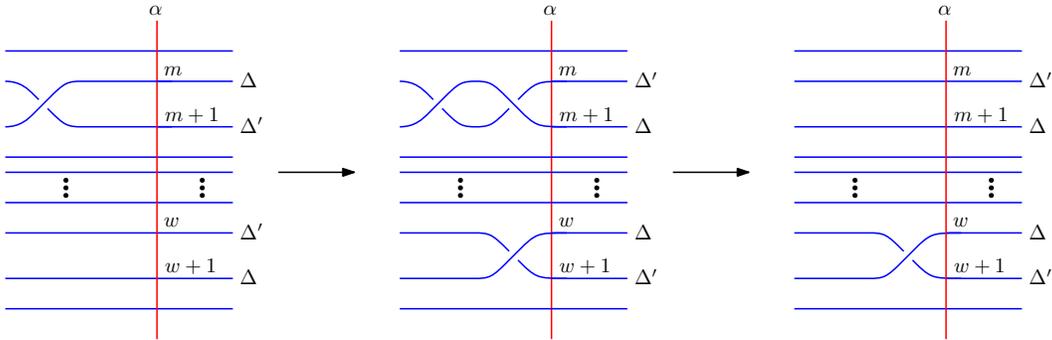} 
\caption{A crossing exchange combined with a Reidemeister 2 move.}
\label{fig:crossingexchangestemr2}
\end{figure}

\section{Proof of Theorem~\ref{thm:maintheorem}} \label{ref:proofsection}

Having established a set of moves on stem diagrams which correspond to the moves of trivial petal addition and crossing exchanges, we now proceed to prove Theorem~\ref{thm:maintheorem}.  We begin with some necessary lemmas.

Lemma~\ref{lem:planarisotopy} shows that given a diagram $D$, up to trivial petal addition and deletion, crossing exchange, and planar isotopy, the choice of axis $\alpha$ only depends on which edge of $D$ (thought of as a 4-valent graph) the basepoint $s$ is located on.  We will say that two points $s,s' \in D$ are on the \emph{same edge of $D$} if $s'$ can be isotoped to $s$ along $D$ without passing through any crossings of $D$.

\begin{lemma}
\label{lem:planarisotopy}
Suppose that $(D , \alpha, s)$ and $(D, \alpha', s')$ are two stem diagrams, and that $s$ and $s'$ are on the same edge of $D$.  Then there is a finite sequence of trivial petal additions and deletions, crossing exchanges, and planar isotopies, which converts the stem diagram $(D , \alpha', s')$ into a stem diagram $(D, \alpha, s'')$, with axis $\alpha$, and some basepoint $s'' \in D \cap \alpha$ which is on the same edge as both $s$ and $s'$.
\end{lemma}

\begin{proof}
Begin by assuming that the diagram has been isotoped so that the axis $\alpha$ corresponds to the $y$-axis in the $xy$-plane $P$.  %Furthermore, by our assumption on $s$ and $s'$ we can isotopy the diagram $(D , \alpha', s')$, while fixing $D$ as a set, so that $s$ and $s'$ coincide.  
We also assume that the axes $\alpha$ and $\alpha'$ agree on the complement of some large disk open $B$ which contains the diagram $D$, and that on the disk $B$ the axes $\alpha$ and $\alpha'$ intersect transversely in a finite collection of $k$ transverse double points.

Each axis $\alpha$ and $\alpha'$ divides the plane $P$ into two regions, a left half-plane and a right half-plane.  Denote the left and right half-planes for $\alpha$ by $L$ and $R$ respectively, and the corresponding half-planes for $\alpha'$ by $L'$ and $R'$.  Notice that the side of $P$ on which a crossing lies is determined by the order in which its strands are traversed when starting from the basepoint.  Hence, by our assumption that $s$ and $s'$ are not separated by any crossings of the diagram $D$, and that both $(D, \alpha,s)$ and $(D, \alpha',s')$ are stem diagrams, each crossing of $D$ must lie in either $L \cap L'$ or in $R \cap R'$.

Suppose now that $k > 0$.  Then there is some connected region $U$ in $P \backslash (\alpha \cup \alpha')$ whose boundary is a bigon consisting of one arc from $\alpha$ and one arc from $\alpha'$, which intersect only at their endpoints $p$ and $p'$.  If the number $k$ of intersection points of $\alpha$ and $\alpha'$ in $B$ is positive, then we can always find such a bigon with at most one of $p$ or $p'$  on the boundary of the closure of ${B}$.

Suppose first that $U$ does not contain a crossing of $D$.   Then after a sequence of trivial petal additions to both $(D, \alpha, s)$ and $(D, \alpha', s')$ we can assume that $D$ intersects $U$ in a collection of parallel arcs, each of which has one endpoint on $\alpha$ and one endpoint on $\alpha'$ (see Figure~\ref{fig:boundarybigonremoval}, where we indicate a possible location for one of the basepoints).  Then there is a planar isotopy of the diagram $(D, \alpha', s')$ which fixes $D$ setwise, and removes a pair of intersection points from $\alpha \cap \alpha'$.  If one of the intersection points was on the boundary $\partial \overline{B}$, then we only remove one intersection point from $\alpha \cap \alpha'$ as in Figure~\ref{fig:boundarybigonremoval}.

\begin{figure}
\includegraphics[width=0.75\textwidth]{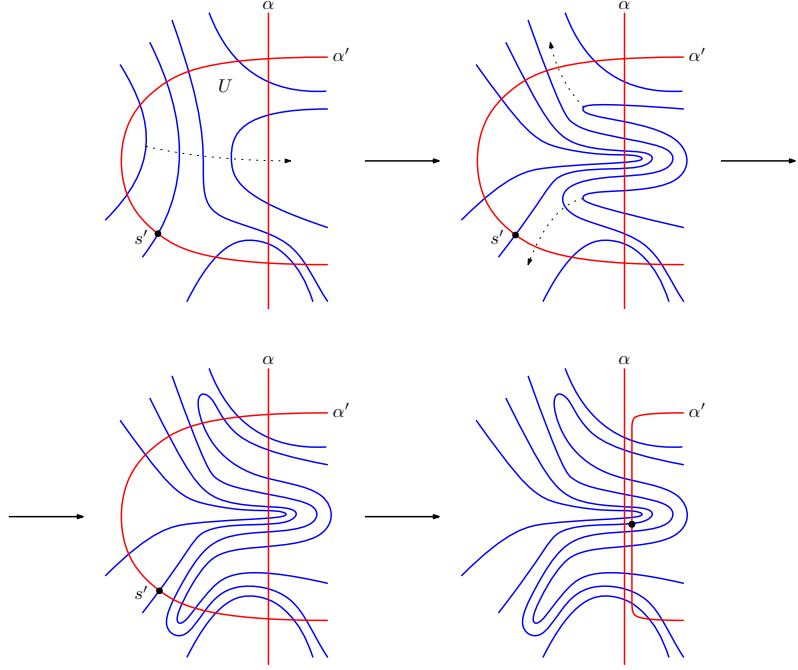} 
\caption{Removing a crossing-free bigon bounded by $\alpha \cup \alpha'$.}
\label{fig:bigonremoval}
\end{figure}

\begin{figure}
\includegraphics[width=0.65\textwidth]{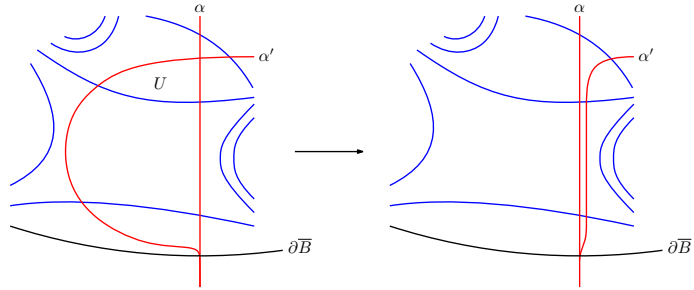} 
\caption{Removing a bigon which intersects the boundary of $\overline{B}$.}
\label{fig:boundarybigonremoval}
\end{figure}

Suppose now that the region $U$ contains a crossing $c$ of the diagram $D$.  Note then that $U$ must be contained in $L \cap L'$ or $R \cap R'$.  Suppose without loss of generality that $U \subset L \cap L'$ (the proof of the case when $U \subset R \cap R'$ is identical).  Let $\gamma$ be a path inside $L$ from the intersection $c$ to a point $q$ in a different component of $L \cap L'$, and which is transverse to $\alpha'$.  For example, we could chose $q$ to be a point in $L$ near the boundary of $\partial \overline{B}$.  Notice that $\gamma$ will be disjoint from $\alpha$ but will intersect $\alpha'$ in an even number of points.  Using trivial petal additions and crossing exchanges we can push the crossing $c$ along $\gamma$ to the point $q$.  Each crossing exchange allows us to push the crossing $c$ through $\alpha'$ twice.  We show how this is done in Figure~\ref{fig:crossingalpha}.

\begin{figure}
\includegraphics[width=\textwidth]{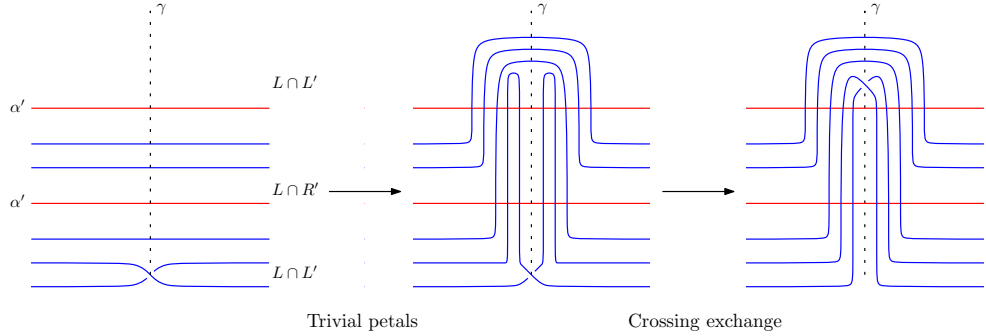} 
\caption{Using trivial petal additions and crossing exchanges to push the crossing $c$ along $\gamma$.}
\label{fig:crossingalpha}
\end{figure}

After pushing all of the crossings outside of $U$ we may again use $U$ to eliminate crossings of $\alpha$ and $\alpha'$ as described above.  We repeat this procedure until the $\alpha$ and $\alpha'$ no longer intersect inside of $B$.  As $\alpha$ and $\alpha'$ agree on $\partial \overline{B}$ however, they will bound a bigon in $B$, which will not contain any crossings of $D$.   Using trivial petal additions, deletions, and isotopies we may then isotope $\alpha'$ across this bigon until $\alpha$ and $\alpha'$ agree.  Notice that during all of these procedures the edges of $D$ on which the basepoints $s$ and $s'$ reside does not change, and hence the repositioned basepoint $s'$, which we denote by $s''$, will be on the same edge as both $s$ and $s'$.
\end{proof}

In the following lemma, we will say that an intersection point $s'$ in $\alpha \cap D$ is \emph{left-pointing} if, when starting at $s'$ and traveling along $D$ in the direction of its orientation, we pass first into the left half-plane $L$, before passing into the right half-plane $R$.  We say that intersection points in $\alpha \cap D$ which are not left-pointing are instead \emph{right-pointing}.

\begin{lemma}
\label{lem:basepoint1}
Let $(D, \alpha, s)$ be a stem diagram with petal permutation $\sigma$, and let $s'$ be any point on $D\cap \alpha$ which is left-pointing.  Then there is a choice of axis $\alpha'$  which agrees with $\alpha$ on a neighborhood of $s'$ and which makes $(D, \alpha',s')$ a stem diagram, such that the petal permutation $\sigma'$ of  $(D, \alpha',s')$ differs from $\sigma$ by a sequence of petal moves.  If $s$ and $s'$ are on the same edge of $D$, then we can take $\alpha'=\alpha$.
\end{lemma}

\begin{proof}
By adding trivial petals, we may assume that each strand of the diagram $(D, \alpha, s)$ is involved in at most one crossing.  Furthermore, we may assume that each crossing is as in Figure~\ref{fig:simplecrossing}.  More precisely, we assume that the four edges of $D$ which are involved in the crossing bound three distinct triangles with the axis $\alpha$, and that the interiors of each of these triangles are disjoint from the diagram $D$ (i.e. there are no other strands nested inside the configuration in Figure~\ref{fig:simplecrossing}).  See Figure~\ref{fig:twostranddiagram} for an illustration of how this is done.  We call the strands that are involved in crossings as in Figure~\ref{fig:simplecrossing} \emph{crossing strands}, and the strands that are not involved in any crossing \emph{trivial strands}.  By adding trivial petals we can also assume that any point in $D \cap \alpha$ is the endpoint of at most one crossing strand.

\begin{figure}
\includegraphics[width=0.08\textwidth]{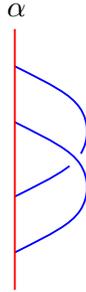} 
\caption{A simple crossing in the stem diagram $(D, \alpha, s)$.}
\label{fig:simplecrossing}
\end{figure}

\begin{figure}
\includegraphics[width=0.5\textwidth]{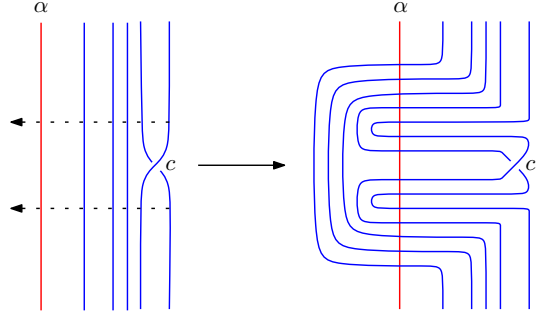} 
\caption{Adding trivial petals so that each crossing is as in Figure~\ref{fig:simplecrossing}.}
\label{fig:twostranddiagram}
\end{figure}

As we traverse the diagram $D$, we will pass points in $D \cap \alpha$ which alternate between left-pointing and right-pointing.  It suffices then to prove the lemma in the case when $s'$ is the first left-pointing intersection we encounter when starting at $s$ and traveling along $D$ in the direction of the positive orientation.  In this case $s$ and $s'$ are separated by two strands, a left-strand $\ell$ and a right-strand $r$.  We first consider the case when both $\ell$ and $r$ are trivial strands.  We illustrate the six possible cases in Figure~\ref{fig:cases}.

\begin{figure}
\includegraphics[width=0.7\textwidth]{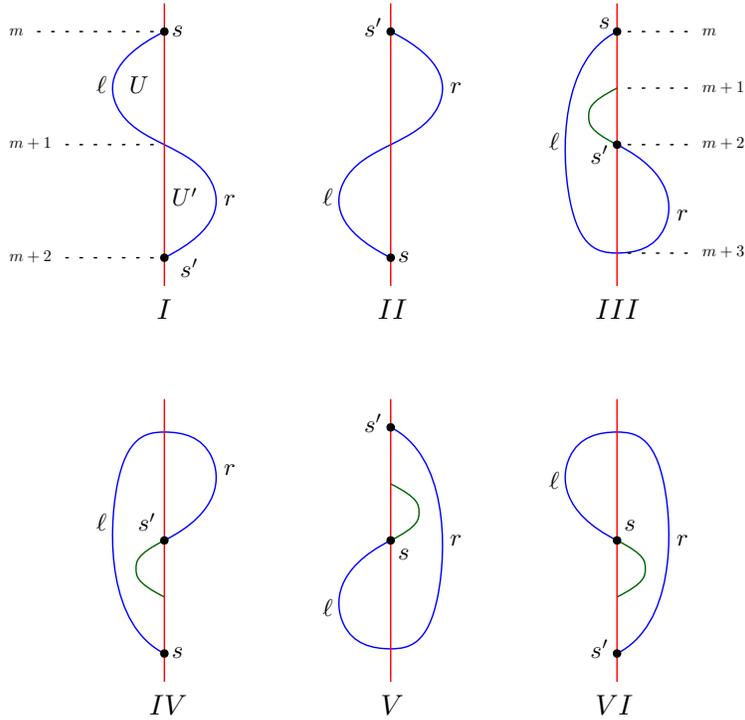} 
\caption{Six possible configurations for trivial strands $\ell$ and $r$ separating $s$ and $s'$.}
\label{fig:cases}
\end{figure}

Before discussing these cases individually, note that in each situation the strands $\ell$ and $r$ each form part of the boundary of a bigon, with the other segment of the boundary coming from $\alpha$.  We will call these bigons $U$ and $U'$ respectively (see Figure~\ref{fig:cases} where we have labelled $U$ and $U'$ only in the first diagram).  Our first step in each of the six cases will be to remove as many of the intersections of the diagram $D$ with the interiors of $U$ and $U'$ as possible.  This can be accomplished using trivial petal addition and deletion, and crossing exchanges as in the proof of Lemma~\ref{lem:planarisotopy}.  In Cases I and II above we see that we can remove all intersections of $D$ with the interiors of $U$ and $U'$, though in Cases III-VI there will necessarily be one edge of intersection of $D$ with $\operatorname{int}{U}$ or $\operatorname{int}{U'}$ which cannot be removed by this procedure.  These edges are shown in Figure~\ref{fig:cases} in green.

Suppose that the diagram $D$ intersects the axis in a number of points whose levels we denote by $t_0, \ldots, t_{2n+1}$, where only a subset of these levels are shown in each case in Figure~\ref{fig:cases}.  Here the levels are written in order, starting at the level $t_0$ of the basepoint $s$, and continuing on in the order they are arrived at when traveling along $D$.  Recall that to find the petal permutation of the stem diagram, we remove the level $t_0$ of the basepoint $s$, and then shift all of the levels that were greater than $t_0$ by one.  To see the effect that shifting the basepoint has on a petal permutation $\sigma=(p_0p_1\cdots p_{2n})$, we must first insert the level $t_0$ of the basepoint $s$ into the word $p_0p_1\cdots p_{2n}$, shifting all of the levels greater than or equal to $t_0$ up by one.  We then remove the level $t_j$ of the basepoint $s'$ from the resulting word, shifting all of the levels greater than $t_j$ down by one.

In each of the cases from Figure~\ref{fig:cases} we are shifting the basepoint from $s$ to $s'$, and there are no intersections $D \cap \alpha$ in the neighborhoods shown besides the ones illustrated.  In other words, all of the intersection points of $D \cap \alpha$ not shown in the diagrams in Figure~\ref{fig:cases} live outside the interval $[t_0,t_j]$.  Hence when changing the basepoint from $s$ to $s'$ the shifts from inserting the level $t_0$ and removing $t_j$ cancel, and so the corresponding letters in the word $\sigma = (p_0p_1 \cdots p_{2n})$ do not change.  It suffices then to only consider how these shifts affect the intersection points shown in Figure~\ref{fig:cases}.  In each diagram we denote the levels of the intersection points shown (starting from the top-most) by $m, m+1, m+2$ and (for the bottom-most point in Cases III-VI) $m+3$.

Case I:  Starting at and including the point $s$ the axis is intersected at levels $m$, $m+1$, and finally $m+2$.  After removing the level of $s$ and shifting, this corresponds to a petal permutation of the form $(m(m+1)p_2 \cdots p_{2n})$.  If instead we build the petal permutation using the basepoint $s'$, we will obtain the petal permutation $(p_2 \cdots p_{2n}m(m+1))$.  Hence changing from the basepoint $s$ to the basepoint $s'$ does not affect the petal permutation in this case.  Case II is handled similarly.

Case III: Starting at and including the point $s$, the axis is intersected at levels $m$, $m+3$, $m+2$, $m+1$.  This gives rise to a petal permutation of the form $((m+2)(m+1)mp_3 \cdots p_{2n})$.  If instead we start at $s'$ we obtain $((m+1)p_3 \cdots p_{2n}m(m+2))=(m(m+2)(m+1)p_3 \cdots p_{2n})$.  Notice, however, that both of these permutations are the same after removing the trivial petal $(m+2)(m+1)$ from each.

Case IV:  Using the basepoint $s$ we obtain a petal permutation $(m(m+1)(m+2)p_3 \cdots p_{2n})$, while using the basepoint $s'$ gives a petal permutation $((m+1)p_3 \cdots p_{2n} (m+2)m)$ $=((m+2)m(m+1)p_3 \cdots p_{2n} )$.  These give rise to the same permutation after removing the trivial petal $m(m+1)$ from each.

Case V: The basepoint $s$ gives a petal permutation $((m+2)mp_2\cdots p_{2n-1}(m+1))=(p_2\cdots p_{2n-1}(m+1)(m+2)m)$, while the basepoint $s'$ gives $(p_2\cdots p_{2n-1}m(m+1)(m+2))$.  Removing the trivial petal $(m+1)(m+2)$ from each of these gives the same permutation.

Case VI:  Using the basepoint $s$ yields a petal permutation $(m(m+2)p_2\cdots p_{2n-1}(m+1))=(p_2\cdots p_{2n-1}(m+1)m(m+2))$, while the basepoint $s'$ gives $(p_2\cdots p_{2n-1}(m+2)(m+1)m)$.  Removing the trivial petal $(m+1)m$ from each of these gives the same permutation.

Note that in each of these cases the axis we choose after the basepoint move can be chosen to be the same as the original axis $\alpha$, from which the last statement in the lemma follows.

We now consider the situation where precisely one of the strands $\ell$ or $r$ which separate $s$ and $s'$ is a crossing strand.  By our assumption on the positioning of the crossing strands at the beginning of this proof, we can assume that in a neighborhood of the basepoint $s$ the stem diagram matches one of the diagrams in Figure~\ref{fig:crossingcases}, and that there are no other parts of $D$ intersecting this neighborhood.  There are four cases to consider, in which we will show that we can move the basepoint from $s_j$ to $s'_j$ without changing either the diagram $D$ or the petal permutation associated to the stem diagram.  The other cases can be obtained by reflection in the vertical direction.

\begin{figure}
\includegraphics[height=4cm]{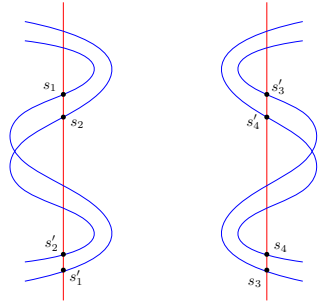} 
\caption{Four possible configurations for a single crossing strand separating $s$ and $s'$.}
\label{fig:crossingcases}
\end{figure}

In Figure~\ref{fig:crossingcases} we omit the over/under-crossing information, as this will depend on which $s_j$ we are currently viewing as being our basepoint.  We will consider the cases of moving the basepoint from $s_j$ to $s'_j$, for $1\leq j \leq 4$, in Figures~\ref{fig:crossingcase1}-\ref{fig:crossingcase4}.  Note that when moving the basepoint through a crossing, the order in which that crossing's strands are traversed along $D$ is changed, and hence  the crossing must be moved to the other side of the axis in order to maintain a valid stem diagram.  In Figures~\ref{fig:crossingcase1}-\ref{fig:crossingcase4} we see that in each of the four cases there is a planar isotopy which takes $D$ to a new configuration with the crossing situated on the other side of $\alpha$, the basepoint moved from $s_j$ to $s'_j$, and the relative levels of the intersections of $D \cap \alpha$ along each strand preserved (ignoring the basepoints).  In Figures~\ref{fig:crossingcase1} and \ref{fig:crossingcase3} the move illustrated changes the petal permutation by a trivial petal deletion along the arc without the basepoint, while in Figures~\ref{fig:crossingcase2} and \ref{fig:crossingcase4} the corresponding petal permutations are unchanged.  This completes the proof of the lemma.
 \end{proof}

\begin{figure}
\includegraphics[height=4cm]{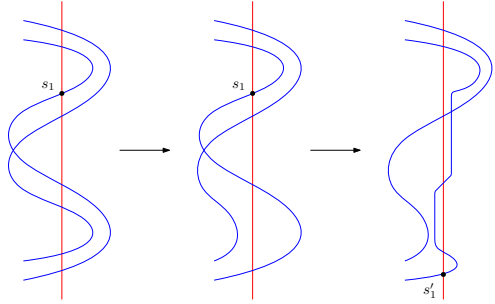} 
\caption{Moving the basepoint from $s_1$ to $s'_1$.}
\label{fig:crossingcase1}
\end{figure}

\begin{figure}
\includegraphics[height=4cm]{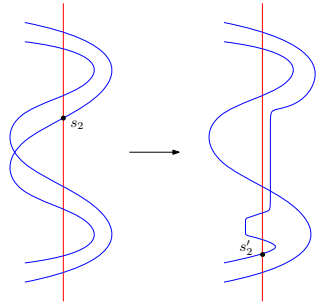} 
\caption{Moving the basepoint from $s_2$ to $s'_2$.}
\label{fig:crossingcase2}
\end{figure}

\begin{figure}
\includegraphics[height=4cm]{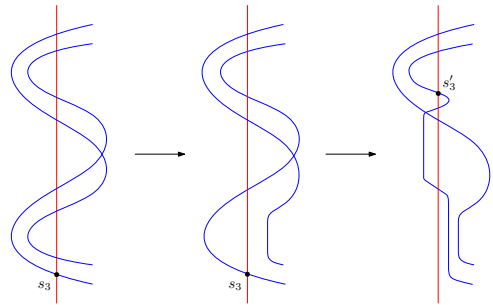} 
\caption{Moving the basepoint from $s_3$ to $s'_3$.}
\label{fig:crossingcase3}
\end{figure}

\begin{figure}
\includegraphics[height=4cm]{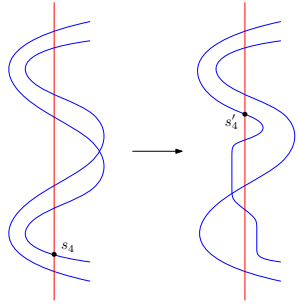} 
\caption{Moving the basepoint from $s_4$ to $s'_4$.}
\label{fig:crossingcase4}
\end{figure}

\begin{figure}
\includegraphics[width=\textwidth]{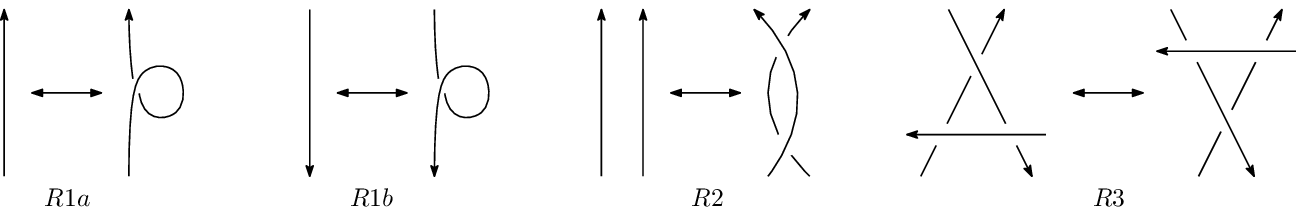} 
\caption{A generating set of Reidemeister moves.}
\label{fig:reidemeister}
\end{figure}

\begin{proposition}
\label{prop:diagram}
Let $(D, \alpha, s)$ and $(D, \alpha' , s')$ be two stem diagrams with associated petal permutations $\sigma$ and $\sigma'$.  Then $\sigma$ and $\sigma'$ are related by a sequence of petal moves.
\end{proposition}

\begin{proof}
By trivial petal additions and Lemma~\ref{lem:basepoint1} we can assume that the basepoints $s$ and $s'$ are on the same edge of $D$.  Then by Lemma~\ref{lem:planarisotopy} there is a point $s'' \in D \cap \alpha$ such that $(D, \alpha, s'')$ is a stem diagram whose associated petal permutation $\sigma''$ is related to $\sigma'$ by petal moves.  Furthermore, $s''$ will be on the same edge as $s$ and $s'$.  But by Lemma~\ref{lem:basepoint1} then the permutation $\sigma''$ is also related to $\sigma$ by such a sequence of petal moves, which completes the proof.
\end{proof}

 Consider now the set of oriented Reidemeister moves in Figure~\ref{fig:reidemeister}.  By \cite{polyak2010minimal} these moves form a generating set for the collection of all oriented Reidemeister moves.

\begin{lemma}
\label{lem:reidemeister}
Suppose that $D$ and $D'$ are diagrams which are related by a single Reidemeister move in Figure~\ref{fig:reidemeister}.  Then there are axes $\alpha$ and $\alpha'$, and basepoints $s \in D \cap \alpha$ and $s' \in D' \cap \alpha'$, such that $(D, \alpha, s)$ and $(D', \alpha',s')$ are stem diagrams with the same petal permutation.
\end{lemma}

\begin{proof}
We begin with the $R1$ and $R2$ moves.  For each of these moves we can select an axis $\alpha=\alpha'$ and basepoint $s=s'$ which are in the complement of a neighborhood of the Reidemeister move under consideration.  These axes and basepoints are illustrated in Figure~\ref{fig:firsttworeidemeister}.  In each case the arc shown can be extended to a full axis $\alpha$, such that $(D,\alpha,s)$ and $(D', \alpha, s)$ are both stem diagrams.  Indeed, given a basepoint $s \in D$ any crossing of $D$ can be labelled as an overcrossing or undercrossing, depending on whether the overcrossing strand or undercrossing strand is encountered first when starting from $s$ and traveling along $D$ in the positively oriented direction.  Any choice of $\alpha$ which separates overcrossings to the right and undercrossings to the left will yield a valid stem diagram (recall that the diagram $D$ is always oriented to the left of $\alpha$ at the basepoint $s$).  Since our choices of $\alpha$ and $s$ satisfy this separation condition locally, $\alpha$ can be extended as required in each case.  Given such an extension, the petal permutations for $(D, \alpha,s)$ and $(D', \alpha, s)$ will be the same.   

\begin{figure}
\includegraphics[width=\textwidth]{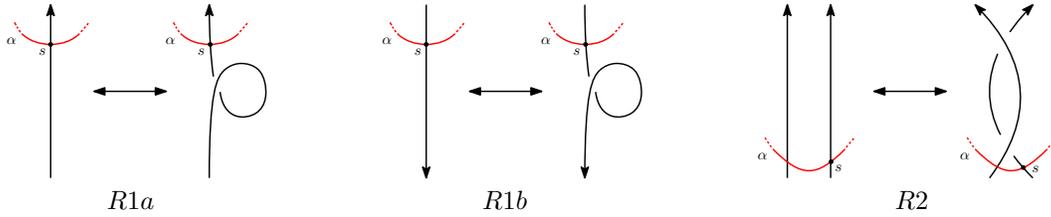} 
\caption{Choices of $\alpha$ and $s$ which avoid the $R1$ and $R2$ moves.}
\label{fig:firsttworeidemeister}
\end{figure}

We must split the final case into two subcases, depending on how the strands involved in the $R3$ move are connected outside the neighborhood illustrated.  These subcases are shown in Figure~\ref{fig:thirdreidemeister}.  In the first subcase we can select an axis $\alpha=\alpha'$ and basepoint $s=s'$ away from the support of the $R3$ move, and proceed as above.  In the second case we choose different axes and basepoints in $D$ and $D'$ as shown on the right of Figure~\ref{fig:thirdreidemeister}.  If the extensions of these two arcs outside of the neighborhood agree, then we can verify that the resulting stem diagrams $(D, \alpha, s)$ and $(D' , \alpha',s')$ yield the same petal permutations as required. 
\begin{figure}
\includegraphics[width=\textwidth]{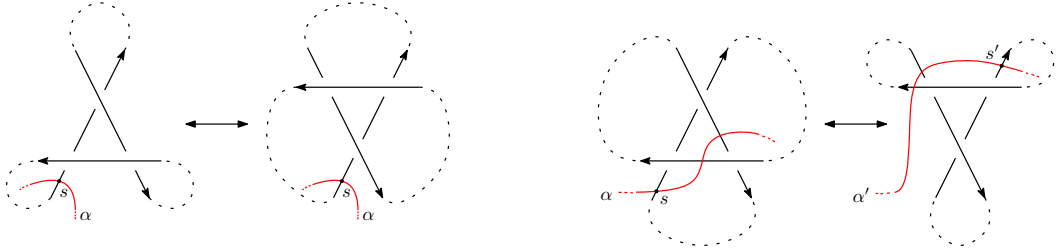} 
\caption{Choosing $\alpha$ and $s$ around $R3$ moves.}
\label{fig:thirdreidemeister}
\end{figure}
\end{proof}

\begin{proof}[Proof of Theorem~\ref{thm:maintheorem}]
Suppose that $\sigma$ and $\sigma'$ are two petal permutations that represent a knot $K$.  Let $(D, \alpha, s)$ be a stem diagram with associated petal permutation $\sigma$, and let $(D', \alpha', s')$ be a stem diagram with associated petal permutation $\sigma'$.  Then as $D$ and $D'$ are both diagrams for $K$, we can find a sequence of diagrams $D=D_0, D_1, D_2, \ldots, D_k=D'$, such that for each $0 \leq j \leq k-1$, the diagram $D_{j+1}$ is obtained from the diagram $D_j$ by a planar isotopy or a single Reidemeister move from Figure~\ref{fig:reidemeister}.   We will show by induction that for $1 \leq j \leq k$ there is a choice of axis $\alpha_j$ and basepoint $s_j$ such that $(D_j,\alpha_j, s_j)$ is a stem diagram, and that the associated  petal permutation $\sigma_j$ is related to $\sigma_{j-1}$ by a sequence of trivial petal additions and deletions, and crossing exchanges.

Set $\alpha_0 = \alpha$ and $s_0 = s$, from which it follows that $\sigma_0 = \sigma$.  Suppose now that for some $0 \leq j \leq k-1$ we have an axis $\alpha_j$ and basepoint $s_j$ so that $(D_j, \alpha_j, s_j)$ is a stem diagram with petal permutation $\sigma_j$.  Then if $D_{j+1}$ is obtained from $D_j$ by a planar isotopy $\varphi_t:P \rightarrow P$, where $\varphi_0\equiv \mathrm{id}_P$ and $\varphi_1(D_j) = D_{j+1}$, we set $\alpha_{j+1} = \varphi_1(\alpha_j)$ and $s_{j+1} = \varphi_1(s_j)$.  It follows then that $(D_{j+1}, \alpha_{j+1}, s_{j+1})$ will be a stem diagram with petal permutation $\sigma_{j+1}=\sigma_j$.

Suppose instead that $D_{j+1}$ is obtained from $D_j$ by a single Reidemeister move from Figure~\ref{fig:reidemeister}.  Then by Lemma~\ref{lem:reidemeister} there are choices of axes $\widetilde{\alpha}_j$ and $\widetilde{\alpha}_{j+1}$, and basepoints $\widetilde{s}_j$ and $\widetilde{s}_{j+1}$ such that $(D_j , \widetilde{\alpha}_j, \widetilde{s}_j)$ and $(D_{j+1} , \widetilde{\alpha}_{j+1}, \widetilde{s}_{j+1})$ are both stem diagrams with the same petal permutation, $\widetilde{\sigma}_j=\widetilde{\sigma}_{j+1}$.   Set $\alpha_{j+1} = \widetilde{\alpha}_{j+1}$ and $s_{j+1}=\widetilde{s}_{j+1}$, whence $\sigma_{j+1} = \widetilde{\sigma}_{j+1}=\widetilde{\sigma}_{j}$.  By Proposition~\ref{prop:diagram} then $\sigma_j$ is related to $\widetilde{\sigma}_{j}=\sigma_{j+1}$ by a sequence of petal moves.  Moreover, by Proposition~\ref{prop:diagram} the final petal permutation in this sequence $\sigma_k$ is also related to $\sigma'$ by a sequence of petal moves, which thereby completes the proof of Theorem~\ref{thm:maintheorem}.
\end{proof}

\bibliographystyle{plain}
\bibliography{bibliography}

\end{document}